\documentclass{amsart}
\usepackage{pgf,pgfarrows,pgfnodes,pgfautomata,pgfheaps,pgfshade}
\usepackage{
graphicx, amsfonts, amssymb, amsbsy, amsmath, amsthm}

\theoremstyle{plain}
\newtheorem{theorem}{Theorem}[section]

\newtheorem{proposition}[theorem]{Proposition}

\theoremstyle{definition}
\newtheorem{definition}[theorem]{Definition}
\newtheorem{example}[theorem]{Example}

\theoremstyle{remark}
\newtheorem{remark}{Remark}

\begin{document}

\title[Exact tests to compare contingency tables under qi and qs]{Exact tests to compare contingency tables under quasi-independence and quasi-symmetry}

\author{C. Bocci}
\address{Department of Information Engineering and Mathematics, University of Siena, Siena, Italy}
\email{cristiano.bocci@unisi.it}
\author{F. Rapallo}
\address{Department of Science and Technological Innovation, University of Piemonte Orientale, Alessandria, Italy}
\email{fabio.rapallo@uniupo.it}

\begin{abstract}
In this work we define log-linear models to compare several square contingency tables under the quasi-independence or the quasi-symmetry model, and the relevant Markov bases are theoretically characterized. Through Markov bases, an exact test to evaluate if two or more tables fit a common model is introduced. Two real-data examples illustrate the use of these models in different fields of applications.
\end{abstract}


\maketitle

\section{Introduction}

Complex models for contingency tables have received an increasing interest in the last decades from researchers and practitioners in different fields, from Biology to Medicine, from Economics to Social Science. For a general introduction to the statistical models for contingency tables see for instance \cite{agresti:13}, \cite{bishop:75} and \cite{kateri:14}. Quasi-symmetry and quasi-independence models are well known log-linear models for square contingency tables. Starting from Caussinus in \cite{caussinus:65}, several authors have considered such models from the point of view of both theory and applications, and it is impossible to give a complete account on all the papers where quasi-independence and quasi-symmetry are studied or used in data analysis. In the next section we will recall the basic facts on the quasi-independence and quasi-symmetry models, while for a full presentation and an historical overview the reader can refer to \cite{bishop:75} and \cite{goodman:02}. Quasi-symmetry is also the topic of a special issue of the {\it Annales de la Facult\'e des Sciences de {T}oulouse}, edited in 2002 by S. Fienberg and P. G. M. van der Heijden \cite{fienberg:02}.

Within Algebraic Statistics, quasi-independence and quasi-symmetry are very important models for contingency tables, for several reasons. We briefly review why the synergy between Algebraic Statistics and quasi-independence has been fruitful. Firstly, Algebraic Statistics provides and exact goodness-of-fit test based on the Diaconis-Sturmfels algorithm. Such test is very flexible when applied to complex models, and it allows us to make exact inference also outside the basic independence model, where the classical Fisher's exact test is available. When the sample size is small, the use of the asymptotic tests based on the chi-square approximation of the test statistics may lead to wrong conclusions, and this fact is even more relevant in this kind of models, where the asymptotics fails also with moderately large sample sizes, see an example in \cite{rapallo:03}. Secondly, under quasi-independence it is possible to fix the diagonal counts, or even to analyze incomplete tables where the diagonal counts (or an arbitrary subset of cells) are undefined or unavailable. To include structural zeros in the analysis, the notion of toric statistical model is a generalization of log-linear model that permits to study also the boundary. Toric models are described by non-linear polynomials, but in several cases it is possible to describe the geometry of such models, or at least it is possible to write their invariants through Computer Algebra systems. Quasi-independence and quasi-symmetry from the point of view of Algebraic Statistics can be found in \cite{rapallo:03}, \cite{aoki:05}, \cite{drton|sturmfels|sullivant:09}, and \cite{aoki:12}. Applications of quasi-symmetry to the problem of rater agreement in biomedical experiments are presented in \cite{rapallo:05}.

In this paper, we use classical techniques from Algebraic Statistics in order to compare several contingency tables under the quasi-independence and quasi-symmetry models. This is accomplished by the construction of a three-way table and by the definition of suitable log-linear models in order to determine if two or more tables fit a common quasi-independence (resp., quasi-symmetry) model, versus the alternative hypothesis that each table follows a specific quasi-independence (resp., quasi-symmetry) model with its own parameters. A third model is also introduced, as its matrix representation is a well known object in Combinatorics, namely the Lawrence lifting of a matrix. For the first two models, the relevant Markov bases are computed theoretically using a distance-reducing argument, while for the third model the Markov bases are characterized only in the case of two tables, and some advices are presented to efficiently run the exact test in the general case. We show two applications of these models on datasets coming from different areas: the first example comes from a rater agreement problem in a biomedical experiments, while the second one concerns the analysis of social mobility tables.

This research suggests several the future directions. From the point of view of Algebra, it would be interesting to study such models when the starting model  on each layer is different from the quasi-independence  and quasi-symmetry models, also including the study of their ideals. From the point of view of Statistics, it would be interesting to study the use of this technique to make inference on other measures of mobility based on log-linear models, also including one-sided tests and their semi-algebraic characterization. For an introductory overview of these measures, with several examples from surveys in European countries, refer to \cite{xie92} and \cite{breen07}.

The paper is organized as follows. In Sect.~\ref{recall-sect} we recall some definitions and basic results about log-linear models and toric models, with special attention to quasi-independence and quasi-symmetry. In particular, we collect here several scattered results on the Markov bases for these models. In Sect.~\ref{compa} we show how to define suitable log-linear models to compare two or more square tables. Given a base log-linear model, we define new log-linear models through the specification of their model matrices. For quasi-independence and quasi-symmetry on several tables, the Markov bases are theoretically computed. Sect.~\ref{example} is devoted to the illustration of two real-data examples.

\section{Markov bases for quasi-independence and quasi-symmetry}
\label{recall-sect}

In this section we recall some basic definitions and properties of log-linear models, with special attention to quasi-independence and quasi-symmetry for square two-way tables. A probability distribution on a finite sample space ${\mathcal X}$
with $K$ elements is a normalized vector of $K$ non-negative real
numbers. Thus, the most general probability model is the simplex
\[
\Delta = \left\{(p_{1}, \ldots, p_{K}) \ : \ p_{k} \geq 0 \ , \
\sum_{k=1}^K p_{k} = 1   \right\} \, .
\]
A statistical model  is therefore a subset of
$\Delta$.

A classical example of finite sample space is the case of a multi-way
contingency table where the cells are the joint counts of two or more
random variables with a finite number of levels each. In the case
of square two-way contingency tables, where the sample space is usually
written as a cartesian product of the form ${\mathcal X}=\{1,
\ldots , I\} \times \{1 , \ldots, I \}$ we will use the notation $p_{i,j}$ to ease the readability. In such case, the two categorical variables are denoted with $X$ and $Y$.

In the classical theory of log-linear models, under the Poisson
sampling scheme the cell counts are independent and identically
distributed Poisson random variables with expected values $Np_1, \ldots,
Np_K$, where $N$ is the sample size, and the statistical model
is  constraints on the raw parameters $p_1, \ldots, p_K$. A
model is log-linear if the log-probabilities lie in an affine
subspace of the vector space ${\mathbf R}^K$. Given $d$ real
parameters $\alpha_1, \ldots, \alpha_d$, a log-linear model is
described, apart from normalization, through the equations:
\begin{equation} \label{loglin}
\log (p_k) = \sum_{r=1}^d A_{k,r}\alpha_r
\end{equation}
for $k=1, \ldots, K$, where $A$ is the model matrix (or design matrix,  \cite{pistone|riccomagno|wynn:01}). Exponentiating Eq.~(\ref{loglin}), we obtain the expression of the corresponding
toric model
\begin{equation} \label{toric}
p_k = \prod_{r=1}^d \zeta_r^{A_{k,r}}
\end{equation}
for $k=1, \ldots , K$, where $\zeta_{r} = \exp(\alpha_r)$, $r=1,
\ldots, d$, are new non-negative parameters. Allowing the $\zeta_r$'s to be non-negative (instead of strictly positive) leads us to consider also the boundary of the models, with points having some entries equal to zero. It follows
that the model matrix $A$ is also the matrix
representation of the minimal sufficient statistic of the model. The matrix representation
of the toric models as in Eq.~(\ref{toric}) is widely discussed
in, e.g., \cite{rapallo:07} and
\cite{drton|sturmfels|sullivant:09}. It is easy to see from Eq.~(\ref{loglin}) that different model matrices with the same image as vector space generate the same log-linear model.

In the two-way case, the simplest (and widely studied) log-linear model is the independence model, which models the stochastic independence between the two categorical variables $X$ and $Y$. Its log-linear form is
\begin{equation} \label{indep}
\log(p_{i,j}) = \mu + \alpha_i^{(X)} + \beta_j^{(Y)}
\end{equation}
with the constraints
\begin{equation} \label{const-indep}
\sum_{i=1}^I \alpha_i^{(X)} = 0 \, , \qquad \sum_{j=1}^I \beta_j^{(Y)} = 0 \, .
\end{equation}

Quasi-independence and quasi-symmetry are both derived from the independence mod\-el adding constraints on given subsets of cells (typically, the cells on the main diagonal) and constraints on the symmetry of the table. Although quasi-independence can be defined for general rectangular tables with fixed counts on an arbitrary subset of ${\mathcal X}$, see \cite{aoki:12}, here we restrict our attention to the case of square tables with fixed counts on the main diagonal. The log-linear form of the quasi-independence model is
\begin{equation} \label{qi}
\log(p_{i,j}) = \mu + \alpha_i^{(X)} + \beta_j^{(Y)} + \gamma_{i}\delta_{i,j}
\end{equation}
where $\delta_{i,j}$ is the Kronecker delta. Also in the quasi-independence model the constraints in Eq.~\eqref{const-indep} hold.

The log-linear form of the quasi-symmetry model is
\begin{equation} \label{qs}
\log(p_{i,j}) = \mu + \alpha_i^{(X)} + \beta_j^{(Y)} + \gamma_{i,j}
\end{equation}
with the constraints
\[
\sum_{i=1}^I \alpha_i^{(X)} = 0 \, , \qquad \sum_{j=1}^I \beta_j^{(Y)} = 0 \, , \qquad \gamma_{i,j} = \gamma_{j,i} \, , \ i,j=1, \ldots, I \, .
\]
In Eq.~(\ref{qs}), the $\alpha^{(X)}_i$ are the parameters of the row effect, the $\beta^{(Y)}_j$ are the parameters of the column effect, while the parameters $\gamma_{i,j}$ force the quasi-symmetry. Comparing Equations (\ref{loglin}) and (\ref{qs}) it is easy to explicitly write the model matrix $A_{\rm qs}$ for the quasi-symmetry model. The first non-trivial example of quasi-independence and quasi-symmetry models is the $3 \times 3$ case, and in this first case the two models coincide as log-linear models. A possible choice is reported in Fig.~\ref{fig-mat}.
\begin{figure}[htbp]
\centering
\[
A^t_{\rm qs} = \left(\begin{array}{ccccccccc}1 & 1 & 1 & 1 & 1 & 1 & 1 & 1 & 1  \\
                                             1 & 1 & 1 & 0 & 0 & 0 & 0 & 0 & 0  \\
                                             0 & 0 & 0 & 1 & 1 & 1 & 0 & 0 & 0  \\
                                             0 & 0 & 0 & 0 & 0 & 0 & 1 & 1 & 1  \\
                                             1 & 0 & 0 & 1 & 0 & 0 & 1 & 0 & 0  \\
                                             0 & 1 & 0 & 0 & 1 & 0 & 0 & 1 & 0  \\
                                             0 & 0 & 1 & 0 & 0 & 1 & 0 & 0 & 1  \\
                                             0 & 1 & 0 & 1 & 0 & 0 & 0 & 0 & 0  \\
                                             0 & 0 & 1 & 0 & 0 & 0 & 1 & 0 & 0  \\
                                             0 & 0 & 0 & 0 & 0 & 1 & 0 & 1 & 0  \\
                                             1 & 0 & 0 & 0 & 0 & 0 & 0 & 0 & 0  \\
                                             0 & 0 & 0 & 0 & 1 & 0 & 0 & 0 & 0  \\
                                             0 & 0 & 0 & 0 & 0 & 0 & 0 & 0 & 1 \end{array}   \right) \, .
\]
\caption{\label{fig-mat}The model matrix of the quasi-symmetry model for $I=3$.}
\end{figure}

Notice that in $A^t_{\rm qs}$ each column represents a cell of the table (the cells are ordered lexicographically for convenience), and each row represents a parameter. Analyzing the structure of $A^t_{\rm qs}$, the first $7$ rows of $A^t_{\rm qs}$ form the model matrix of the independence model $A^t_{\rm ind}$, while the last three rows of $A^t_{\rm qs}$ define one real parameter for each diagonal cell, and hence force the diagonal cells to be fitted exactly. This parametrization is redundant, since the model has $1$ degree of freedom, and therefore $8$ parameters are sufficient to describe the model. The ideal of the independence model with model matrix $A_{\rm ind}$ is the set of all $2 \times 2$ minors of the table of probabilities, i.e.,
\begin{multline} \label{minors}
{\mathcal I}_{A_{\rm ind}}= \mathrm{Ideal}
(p_{1,1}p_{2,2}-p_{1,2}p_{2,1}, \ p_{1,1}p_{2,3}-p_{1,3}p_{2,1},
\ p_{1,1}p_{3,2}-p_{1,2}p_{3,1}, \ \\
p_{1,1}p_{3,3}-p_{1,3}p_{3,1}, \ p_{1,2}p_{2,3}-p_{1,3}p_{2,2}, \
p_{1,2}p_{3,3}-p_{3,2}p_{2,3}, \  \\
p_{2,1}p_{3,2}-p_{3,1}p_{2,2}, \ p_{2,1}p_{3,3}-p_{3,1}p_{2,3}, \
p_{2,2}p_{3,3}-p_{3,2}p_{2,3} ) \, ,
\end{multline}
while for the quasi-independence model and for the quasi-symmetry model from the matrix $A_{\rm
qs}$ we have only one binomial:
\[
{\mathcal I}_{A_{\rm qind}}= {\mathcal I}_{A_{\rm qs}} = \mathrm{Ideal}
(p_{1,2}p_{2,3}p_{3,1}-p_{1,3}p_{3,2}p_{2,1} ) \, .
\]
From the ideals above one can easily derive the corresponding Markov bases. Given a model matrix $A$, recall that a move is a table $m$ with integer entries such that $Am=0$, and that a set of moves ${\mathcal M}_A$ is a Markov basis if all fibers
\[
{\mathcal F}_{A,b} = \{ f \in {\mathbb N}^k \ : \ A^t f= b \}
\]
are connected. Following \cite{drton|sturmfels|sullivant:09}, from the point of view of computations, the easiest way to build a Markov basis is to compute the binomials in a system of generators of the toric ideal ${\mathcal I}_A$ of $A$ and to transform such binomials through the logs: $x^{m^+}-x^{m^-} \mapsto \pm m= \pm(m^+-m^-)$. For instance, the ideal ${\mathcal I}_{A_{\rm qs}}$ yields a Markov basis with only two moves:
\[
m=\pm \left(\begin{array}{ccc} 0  & +1 & -1 \\
                           -1 &  0 & +1 \\
                           +1 & -1 & 0
                           \end{array}\right) \, .
\]

To conclude this section, we collect some results on Markov bases for quasi-indepen\-dence and quasi-symmetry models to be found in \cite{aoki:05}, \cite{aoki:12}, \cite{drton|sturmfels|sullivant:09}, and \cite{rapallo-yoshida:10}.

A loop of degree $r$ on ${\mathcal X}$ is an $I
\times I$ a move $m= \pm m_r(i_1, \ldots , i_r ; j_1, . . . , j_r
)$ for $1 \leq i_1, \ldots , i_r \leq I$, $1 \leq j_1,
\ldots , j_r \leq I$, where $m_r (i_1, \ldots , i_r ; j_1, \ldots
, j_r )$ has entries
\begin{equation*}
\begin{array}{c}
m_{i_1,j_1} = m_{i_2,j_2} = \ldots = m_{i_{r-1},j_{r-1}} = m_{i_r,j_r} = 1,\\
m_{i_1,j_2} = m_{i_2,j_3} = \ldots = m_{i_{r-1},j_r} = m_{i_r,j_1} = -1,
\end{array}
\end{equation*}
and all other elements are zero. The indices $i_1, i_2,
\ldots $, are all distinct, as well as the indices $j_1, j_2, \ldots$, i.e.
\begin{equation*}
i_m \not = i_n \mbox{ and } j_m \not = j_ n \mbox{ for all } m
\not = n \, ,
\end{equation*}

A loop of degree $2$, $M_2(i_1, i_2; j_1, j_2)$, is called a basic move, and a loop $m_r$ is called df $1$ if its support does not contain the support of any other loop. A loop $m_r$ is called a symmetric loop if $\{i_1, \ldots ,i_r\} = \{j_1, \ldots, j_r \}$.

The first result concerns quasi-independence with possible structural zeros. Let $\mathcal S$ be the set of structural zeros.

\begin{proposition}\label{aokitakemura}
The set of df $1$ loops of degree $2, \ldots , I$ with support on ${\mathcal X} \setminus {\mathcal S}$
forms a unique minimal Markov basis for $I \times I$
contingency tables under the quasi-independence model with possible structural zeros. When the fixed cells of the table are located only on the main diagonal, the minimal Markov basis is formed by the df $1$ loops of degree $2$ and $3$.
\end{proposition}
For the quasi-symmetry model we have the following

\begin{proposition}\label{drtonexercise}
The set of symmetric loops of degree $3, \ldots , I$ with support outside  the main diagonal form a unique minimal Markov basis for $I \times I$
contingency tables with structural zeros under the quasi-symmetry model. Such set of moves is also a Graver basis.
\end{proposition}

\section{Comparison of several tables under quasi-independence and quasi-symmetry} \label{compa}

As outlined in the Introduction, in this section we define three log-linear models to compare two or more square tables under quasi-independence and quasi-symmetry, and we study the corresponding Markov bases. Let us consider $H$ tables ($H \geq 2$) and define a three-way contingency table $T$ by stacking the $H$ tables. Conversely, each original table is a layer of the table $T$. Let $K'=HI^2$ be the number of cells of $T$. Since the definition of the new models can be done starting form a generic log-linear model on the two-way table, we present the models is a general context, and then we write the explicit log-linear representation in the case of quasi-independence and quasi-symmetry in order to highlight the meaning of such new models in our cases.

\begin{definition}
Let $A$ be the model matrix of a log-linear model. We define three log-linear models for $T$:
\begin{itemize}
\item Under the model $M_0$ we assume that all the layers follow a common model with model matrix $A$;

\item Under the model $M_1$ we assume that each layer of the table $T$ follows a model with model matrix $A$ with its own parameters, without further constraints;

\item Under the model $M_2$ we assume that each layer of the table $T$ follows a model with model matrix $A$ with its own parameters, and with the additional constraint of fixed marginal sums over the layers.
\end{itemize}
\end{definition}

The model matrices of $M_0$, $M_1$ and $M_2$ have a simple block structure. In fact:
{\small{\[
A_{M_0}^t=
\left(
\begin{array}{c|c|c}
A^t & \cdots & A^{t}   \\
\hline
{\mathbf 1}_K &  & \\
 \hline
  & \ddots & \\
 \hline
 &  & {\mathbf 1}_K
\end{array}
\right)
\quad
A_{M_1}^t=
\left(
\begin{array}{c|c|c}
A^t & &   \\
\hline
 & \ddots & \\
 \hline
 &  & A^t
\end{array}
\right)
\quad
A_{M_2}^t=
\left(
\begin{array}{c|c|c}
A^t & &   \\
\hline
 & \ddots & \\
 \hline
 &  & A^t \\ \hline
{\mathbf I}_K & \ddots & {\mathbf I}_K
\end{array}
\right)
\]}}
where ${\mathbf 1}_K$ is a row vector of $1$'s with length $K$, ${\mathbf I}_K$ is a the identity matrix with dimensions $K \times K$ and each empty block means a block filled with $0$'s. The matrix $A_{M_2}^t$ is the $H$-th order Lawrence lifting of $A^t$ and its properties in terms of Markov and Graver bases have been studied in \cite{sturmfels:96} and \cite{santossturmfels}.

Writing explicitly the log-linear form of the three models in the case of quasi-indepen\-dence we have the equations below. For $M_0$:
\begin{equation}
(M_0) \qquad \log(p_{i,j,h}) = \mu + \mu_h + \alpha_{i}^{(X)} + \beta_{j}^{(Y)} + \gamma_{i}\delta_{i,j}
\end{equation}
with the constraint $\sum_{h = 1}^{H} \mu_h = 0$ in addition to the constraints on $\alpha_{i}^{(X)}$ and $\beta_{j}^{(Y)}$ naturally derived from the basic quasi-independence model in Eq.~(\ref{qi}). The second model $M_1$ is defined by
\begin{equation}
(M_1) \qquad \log(p_{i,j,h}) = \mu + \mu_h + \alpha_{h,i}^{(X)} + \beta_{h,j}^{(Y)} + \gamma_{i,h}\delta_{i,j}
\end{equation}
with the constraint $\sum_{h = 1}^{H} \mu_h = 0$ in addition to the constraints on $\alpha_{h,i}^{(X)}$ and $\beta_{h,j}^{(Y)}$ derived from the basic quasi-independence model in Eq.~(\ref{qi}) and valid on each layer of the table $T$. The third model $M_2$ is defined by
\begin{equation}
(M_2) \qquad \log(p_{i,j,h}) = \mu + \mu_h + \mu_{i,j} + \alpha_{i,h}^{(X)} + \beta_{j,h}^{(Y)} + \gamma_{i,h}\delta_{i,j}
\end{equation}
with the same constraints as in $M_1$ plus the additional constraints $\sum_{i=1}^I \mu_{i,j} = 0, j=1, \ldots, I$ and $\sum_{j=1}^I \mu_{i,j} = 0, i=1, \ldots, I$.

In the case of quasi-symmetry we obtain the expressions below. For $M_0$:
\begin{equation}
(M_0) \qquad \log(p_{i,j,h}) = \mu + \mu_h + \alpha_{i}^{(X)} + \beta_{j}^{(Y)} + \gamma_{i,j}
\end{equation}
with the constraint $\sum_{h = 1}^{H} \mu_h = 0$ in addition to the constraints on $\alpha_{i}^{(X)}$, $\beta_{j}^{(Y)}$ and $\gamma_{i,j}$ naturally derived from the basic quasi-symmetry model in Eq.~(\ref{qs}). The second model $M_1$ is defined by
\begin{equation}
(M_1) \qquad \log(p_{i,j,h}) = \mu + \mu_h + \alpha_{h,i}^{(X)} + \beta_{h,j}^{(Y)} + \gamma_{i,j,h}
\end{equation}
with the constraint $\sum_{h = 1}^{H} \mu_h = 0$ in addition to the constraints on $\alpha_{h,i}^{(X)}$, $\beta_{h,j}^{(Y)}$ and $\gamma_{i,j,h}$ derived from the basic quasi-symmetry model in Eq.~(\ref{qs}) and valid on each layer of the table $T$. The third model $M_2$ is defined by
\begin{equation}
(M_2) \qquad \log(p_{i,j,h}) = \mu + \mu_h + \mu_{i,j} + \alpha_{i,h}^{(X)} + \beta_{j,h}^{(Y)} + \gamma_{i,j,h}
\end{equation}
with the same constraints as in $M_1$ plus the additional constraints $\sum_{i=1}^I \mu_{i,j} = 0, j=1, \ldots, I$ and $\sum_{j=1}^I \mu_{i,j} = 0, i=1, \ldots, I$.

With a simple linear algebra argument, it easy to see that $M_0 \subset M_1 \subset M_2$. As a consequence, these models can be used in two ways. They can be applied separately, to define goodness-of-fit test to contrast the observed table with a given model, or the can be used to define a test for nested models, see e.g. \cite{agresti:13} where several examples are introduced and discussed. In the next section we will focus on tests for nested models.

Now, we compute the Markov bases for the models $M_0$, $M_1$ and $M_2$ defined above. Let ${\mathcal M}$ be a Markov basis for quasi-independence (or quasi-symmetry). In the proofs below we will make use of a distance-reducing argument, introduced in \cite{takemura|aoki:05}.

For the model $M_0$, define the following two types of moves $b$:
\begin{itemize}
\item[type 1:] fix a move $m\in \mathcal{M}$ and split it on the different layers with the condition that with the condition that each row of $m$ belongs to one layer.
\item[type 2:] choose integers $1\leq i_1 < i_2\leq I$ and $1\leq j_1 < j_2 \leq I$ and define the moves $\pm b$ where $b$ has zero coordinate except for
\[
\begin{array}{l}
b_{i_1,j_1,h_1}=1\\
b_{i_2,j_2,h_1}=-1\\
b_{i_1,j_1,h_2}=-1\\
b_{i_2,j_2,h_2}=1
\end{array}
\]
for $1\leq h_1<h_2\leq H$.
\end{itemize}
Consider as a set of moves
\[
\mathcal{M}_0=\mathcal{B}_1\cup \mathcal{B}_2
\]
where $\mathcal{B}_i$ is the set of moves of type $i$.

\begin{proposition}
The set ${\mathcal M}_0$ above is a Markov basis for the model $M_0$.
\end{proposition}
\begin{proof}
Let  $v\in \mathbb{Z}^{I\times I \times H}$. Then $v\in \mbox{Ker}_{\mathbb{Z}}(A^t_{M_0})$ if and only
\begin{itemize}
\item[i)] $A^t \left(\sum_{h=1}^H v_{\bullet ,\bullet ,h}\right)=0$ that is  $\sum_{h=1}^H v_{\bullet ,\bullet ,h}\in \mbox{Ker}_{\mathbb{Z}}(A^t)$;
\item[ii)]  $\sum_{i,j=1}^{I} v_{i,j,h}=0$ for all $1\leq h \leq H$
\end{itemize}
where ii) follows directly looking at the vectors of ones $ {\mathbf 1}_K$ in the definition of $A^t_{M_0}$.

Let $u,v$ vectors with same value of the sufficient statistic, i.e. $A_{M_0}^t u =A^t_{M_0} v$. We want to prove that there exists $b\in \mathcal{B}$ such that
\[u+b\geq 0
\]
and
\[
||u+b-v||_1<||u-v||_1.
\]

Since $u$ and $v$ are distinct with $A^t_{M_0}u=A^t_{M_0}v$ then there exists a positive entry in $u-v$, say $u_{i_1,j_1,h_1}-v_{i_1,j_1,h_1}>0$. Suppose that such entry belongs to the main diagonal, i.e., $i_1=j_1$. Then there exists another layer $h_2$ such that $u_{i_1,j_1,h_1}-v_{i_1,j_1,h_1}<0$. Moreover, by the condition ii), there exists a positive entry of $u-v$ in the layer $h_2$ and a negative entry in the layer $h_1$: $u_{i_2,j_2,h_1}-v_{i_2,j_2,h_1}<0$ and $u_{i_3,j_3,h_2}-v_{i_3,j_3,h_2}>0$ for some indices $i_2,i_3,j_2,j_3$. Now consider the move $b \in {\mathcal B}_2$ defined by
\[
\begin{array}{l}
b_{i_1,j_1,h_1}=-1\\
b_{i_1,j_1,h_2}=+1\\
b_{i_3,j_3,h_2}=-1\\
b_{i_3,j_3,h_1}=+1
\end{array}
\]
that satisfies $||u-v||_1>||u+b-v||_1$.

Thus, we can consider only the case where $u-v$ is zero on the main diagonal of all layers. Let $U$ and $V$ be the sum over the $H$ layers of $u$ and $v$, respectively. By condition i), $U-V\in\mbox{Ker}_{\mathbb{Z}}(A^t)$. By Propositions \ref{aokitakemura} and \ref{drtonexercise}, there exists a distance reducing move $m$ which is a df $1$ loop (in the case of quasi-independence) or a symmetric loop (in the case of quasi-symmetry):
\[
||U+m-V||_1 < ||U-V||_1
\]
Let $\mathcal{I}=\{(i_1,j_1),(i_1,j_2),\dots, (i_t,j_t), (i_t,j_1)\}$ be the set of the indices where $m$ has nonzero entries. Without loss of generality we can suppose that $m_{i_1,j_1}=-1$. The move $m$ is defined such that
\[
\begin{array}{l}
m_{\alpha,\beta}= +1 \qquad \mbox {if and only if} \qquad U_{\alpha,\beta}-V_{\alpha,\beta} < 0 \\
m_{\alpha,\beta}= -1 \qquad \mbox {if and only if} \qquad U_{\alpha,\beta}-V_{\alpha,\beta} > 0
\end{array}
\]
for all $(\alpha,\beta) \in {\mathcal I}$, except at most for the last index $(i_t,j_1)$, where $m_{i_t,j_1}=+1$.

Now, we split $m$ in the $H$ layers. For this aim, notice that $U_{\alpha,\beta}-V_{\alpha,\beta} < 0$ implies that there exists a layer $h$ such that $u_{\alpha,\beta,h}-v_{\alpha,\beta,h}<0$, and similarly $U_{\alpha,\beta}-V_{\alpha,\beta} > 0$ implies that there exists a layer $h$ such that $u_{\alpha,\beta,h}-v_{\alpha,\beta,h}>0$. We then split $m$ by putting the $+1$ in the layer with a negative entry of $u-v$ and $-1$ in the layer with a positive entry of $u-v$. This can be done for all $(\alpha,\beta) \in {\mathcal I} \setminus {(i_t,j_1)}$. The last $+1$ can be assigned to an arbitrary layer. Thus we have a matrix $b$ defined by a sequence of indices
\[
{\mathcal I}' = \{(i_1,j_1,h_{11}),(i_1,j_2,h_{12}),\dots, (i_t,j_t,h_{tt}), (i_t,j_1,h_{t1})  \}
\]
Now, we arrange the entries in such a way that the entries in each row belong to the same layer. This will ensures that the split move $b$ satisfies the condition ii) above.

Let us consider the first row and the corresponding indices $(i_1,j_1,h_{11}),(i_1,j_2,h_{12})$ where $b$ is nonzero. By construction, $u_{i_1,j_1,h_{11}} - v_{i_1,j_1,h_{11}}>0$ and $u_{i_1,j_2,h_{12}} - v_{i_1,j_2,h_{12}}<0$. If $h_{11}=h_{22}$ we have concluded, otherwise there is an entry $(\alpha,\beta,h_{12})$ in the layer $h_{12}$ such that $u_{\alpha,\beta,h_{12}}-v_{\alpha,\beta,h_{12}}>0$. Now, consider the preliminary move $b^{(p)} \in {\mathcal B}_2$ defined by
\[
\begin{array}{l}
b^{(p)}_{\alpha,\beta,h_{12}}=-1\\
b^{(p)}_{i_1,j_1,h_{12}}=+1\\
b^{(p)}_{i_1,j_1,h_{11}}=-1\\
b^{(p)}_{\alpha,\beta,h_{11}}=+1 .
\end{array}
\]
If $u_{i_1,j_1,h_{12}}-v_{i_1,j_1,h_{12}}<0$, then the move $b^{(p)}$ is distance-reducing and we have concluded. Otherwise, $||u+b^{(p)}-v||_1 = ||u-v||_1$ and $u_{i_1,j_1,h_{12}}+b^{(p)}_{i_1,j_1,h_{12}}-v_{i_1,j_1,h_{12}}>0$, and therefore we can move the $-1$ into the layer $h_{12}$. Now, the remaining $+1$'s in $b$ can be simply moved, row by row, in the same layer of the corresponding $-1$ and this is enough to conclude: there is a move $b \in {\mathcal B_1}$ such that $||u+b^{(p)}+b-v||_1 < ||u-v||_1$.
\end{proof}

\begin{example}
If $I=3$ and $H=2$ an example of move of type $\mathcal{B}_1$, for both the quasi-independence and the quasi-symmetry models, is drawn in Fig.~\ref{fig_1}.
\begin{figure}[htbp]
\centering
\makebox{\pgfdeclareimage[height=4.56cm, width=4cm]{figM0}{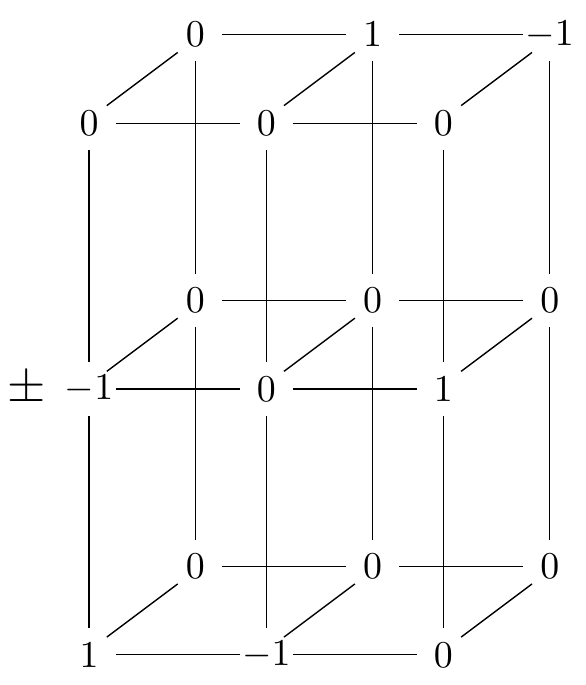}
\pgfuseimage{figM0}}
\caption{\label{fig_1}An example of split move of type ${\mathcal B}_1$ for the model $M_0$.}
\end{figure}
\end{example}

We remark that, while in the quasi-independence model the diagonal cells are fixed, this is no longer true when two or more tables are compared under the model $M_0$. In fact, the moves in ${\mathcal B}_2$ act also on the diagonal cells.

The model $M_1$ is easy to analyze, and the moves are given by the following proposition.

\begin{proposition}
The set of moves
\[
\mathcal{M}_1=\{(m,0,\dots,0), (0,m,0,\dots, 0), \dots, (0,\dots, 0, m) :  m \in {\mathcal M}\}
\]
is a Markov basis for $M_1$.
\end{proposition}

\begin{example}
If $I=3$ and  $H=2$ the moves for the model ${M_1}$ for quasi-independence and quasi-symmetry are drawn in Fig.~\ref{fig_2}.
\begin{figure}[htbp]
\centering
\makebox{\begin{tabular}{ccc}
\pgfdeclareimage[height=4.56cm, width=4cm]{figM1a}{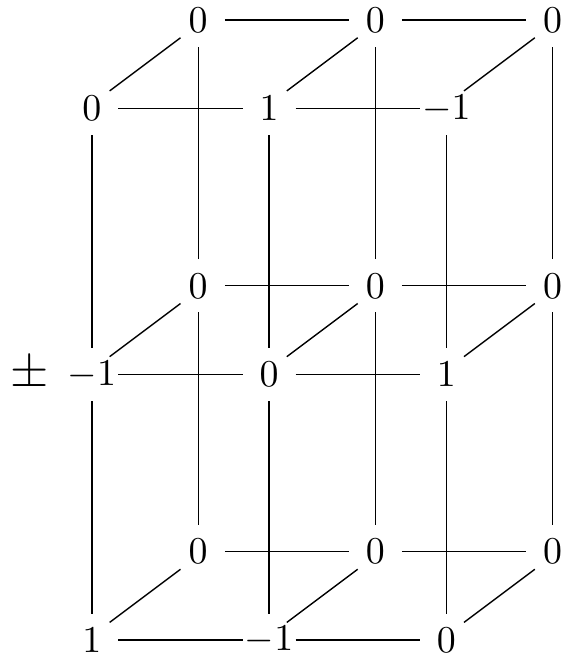}
\pgfuseimage{figM1a}& $\quad$&\pgfdeclareimage[height=4.56cm, width=4cm]{figM1b}{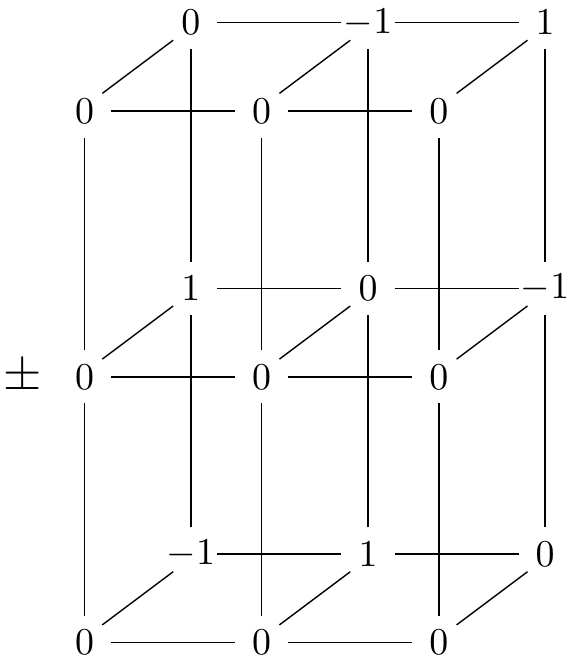}
\pgfuseimage{figM1b}\end{tabular}}
\caption{\label{fig_2}The moves for the model $M_1$.}
\end{figure}
\end{example}

For the model $M_2$, we restrict to the case of two layers. Let us consider a Graver basis ${\mathcal G}$ for the model matrix $A^t$, and consider the set of moves
\[
{\mathcal L}=\left\{
(m, -m) : m \in {\mathcal G}
\right\} \, .
\]
\begin{remark}
For quasi-symmetry, the Markov basis from Proposition \ref{drtonexercise} is also a Graver basis, while in the case of quasi-independence we need to consider the set of all df $1$ loops of degree $2, \ldots, I$.
\end{remark}

\begin{proposition}
The set ${\mathcal L}$ is a Graver basis (and thus also a Markov basis) for the model $M_2$.
\end{proposition}
\begin{proof}
This follows from Theorem $7.1$ in \cite{sturmfels:96}.
\end{proof}

\begin{example}
For example, if $I=3$ the are only two moves for the model $M_2$ for both the quasi-independence and the quasi-symmetry models. They are drawn in Fig.~\ref{fig_3}.
\begin{figure}[htbp]
\centering
\makebox{\pgfdeclareimage[height=4.56cm, width=4cm]{figM2}{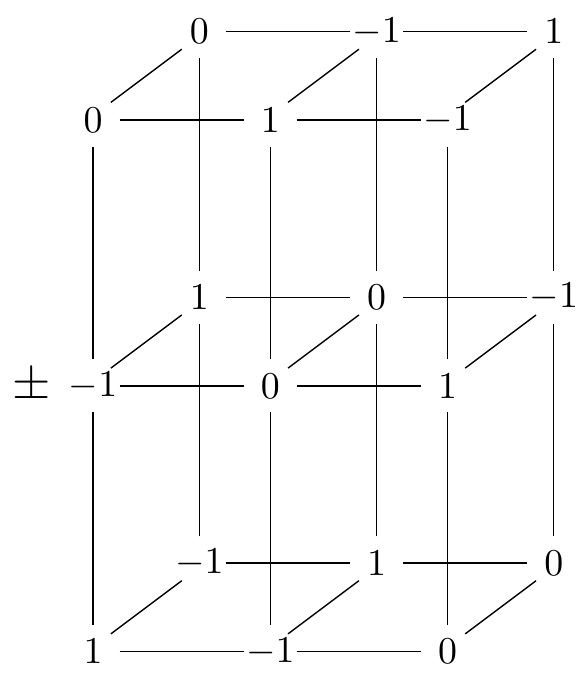}
\pgfuseimage{figM2}}
\caption{\label{fig_3}The moves for the model $M_2$.}
\end{figure}
\end{example}

For more than $2$ layers (i.e., for higher Lawrence configurations), the Markov basis for $M_2$ may be computed through {\tt 4ti2} \cite{hemmecke|hemmecke|malkin:05}, but the number of moves increases rapidly with $I$ and $H$. A valid alternative in this case is to run the Markov chain of the Diaconis-Sturmfels algorithm without a Markov basis, as described in Chapter $16$ of \cite{aoki:12}.

\section{Examples} \label{example}

In this section we present two numerical examples where the comparison of two quasi-symmetry tables may be used. Before introducing the two numerical examples, we briefly recall the Diaconis-Sturmfels algorithm, by adapting the notation to the case of a test for nested models to contrast model $M_0$ inside the model $M_1$. For details on the Diaconis-Sturmfels algorithm, see \cite{drton|sturmfels|sullivant:09}.

Let $f$ be the observed table of counts, and write $f$ as a vector of length $K'$ according to the row labels of $A_{M_0}$. Moreover, let $p$ be the vector of probabilities associated to $f$. The test for nested models with null hypothesis $H_0: \ p \in M_0 \subset M_1$ versus $H_1: \ p \in M_1$ can be done using the log-likelihood ratio statistic
\[
G^2 = 2 \sum_{k=1}^{K'} f_k \log \left( \frac {\hat f_{0k}} {\hat
f_{1k}} \right) \, ,
\]
where $\hat f_{0k}$ and $\hat f_{1k}$ are the maximum likelihood
estimates of the expected cell counts under the models $M_0$ and $M_1$
respectively. In the asymptotic theory, the value of $G^2$ must be compared with the
quantiles of the chi-square distribution with the appropriate number of degrees of freedom, depending on the dimensions of the table.

We use here the Diaconis-Sturmfels algorithm based on a Markov basis ${\mathcal M}_0$ of the model $M_0$. Given the observed table $f$, its reference set under $M_0$ is
\[
{\mathcal F}(f) = \left\{ f' \in {\mathbf N}^{K'} \ : \ A^t_{M_0} f' = A^t_{M_0} f
\right\}
\]
and it is connected using the algorithm below. At each step:
\begin{enumerate}
\item let $f$ be the current table;

\item choose with uniform probability a move $m \in {\mathcal
M}_0$;

\item define the candidate table as $f_+=f + m$;

\item generate a random number $u$ with uniform distribution over
$[0,1]$. If $f_+ \geq 0$ and
\[
\min \left\{ 1 , \frac {\mathcal H(f_+)}  {\mathcal H(f)} \right\}
> u
\]
then move the chain in $f_+$; otherwise stay at $f$. Here ${\mathcal H}$ denotes the hypergeometric distribution on ${\mathcal F}(t)$.
\end{enumerate}
After an appropriate burn-in-period and taking only tables at fixed times to reduce correlation between the sampled tables, the proportion of sampled tables with test statistics greater than or equal to the test statistic of the observed one is the Monte Carlo approximation of
$p$-value of the log-likelihood ratio test. The results in this section are based on Monte Carlo samples of size $10,000$, and the corresponding asymptotic $p$-values are displayed for comparison.

\subsection{Rater agreement data}

The data in Tab.~\ref{tab:1ra} summarize the results of a medical experiment involving the evaluation of agreement among raters. Two independent raters are asked to assign a set of medical images to $5$ different stages of a disease (levels $1$ to $5$ in increasing order of severity of the disease). To check the relevance of a thorough training of the raters, a first set of images has been classified before the training session, while a second set has been classified after the training session.

\begin{table}[htbp]
\caption{\label{tab:1ra}Rater agreement data. Columns represent the grading assigned by the first rater, rows represent the grading assigned by the second rater. The data in the left panel have been collected before the training session, the data in the right panel have been collected after the training session.}
\centering
\begin{tabular}{ccccccccccccc}
\hline\noalign{\smallskip}
 & $1$ & $2$ & $3$ & $4$ & $5$ & $ \qquad $& & $1$ & $2$ & $3$ & $4$ & $5$   \\
\noalign{\smallskip}\hline\noalign{\smallskip}
$1$ & 10 & 2 & 1 & 4 & 0 & & $1$ & 16 & 5 & 1 & 0 & 0 \\
$2$ & 4 & 8 & 4 & 1 & 1 & & $2$ & 0 & 15 & 2 & 1 & 0 \\
$3$ & 0 & 0 & 10 & 3 & 1 & & $3$ & 1 & 3 & 14 & 1 & 1 \\
$4$ & 1 & 1 & 4 & 11 & 0 & & $4$ & 0 & 2 & 0 & 14 & 3 \\
$5$ & 0 & 1 & 3 & 3 & 10 & & $5$ & 0 & 2 & 0 & 3 & 14 \\
\noalign{\smallskip}\hline\noalign{\smallskip}
\end{tabular}
\end{table}

We use the quasi-symmetry model, and the Markov basis ${\mathcal M}_0$ for the model $M_0$ consists of $2 \cdot 1004$ moves (i.e., $1004$ moves, each of them with the two signs). First, we run two exact tests for the goodness-of-fit of the two tables separately under quasi-symmetry, and we obtain exact $p$-values equal to $0.912$ and $0.791$ respectively ($G^2= 1.578$ and $G^2=4.303$ respectively, with $7$ df). Running the test for a unique quasi-symmetry model described in the previous section, the exact test gives a $p$-value equal to $0.029$ ($G^2=30.589$ with $18$ df, corresponding to an asymptotic $p$-value equal to $0.014$). While both the layers fit a quasi-symmetry model very well, they do not fit a common quasi-symmetry model. This means that there are significant differences in the classification of the medical images before and after the training.

\subsection{Social mobility data}

As a second numerical example we consider the data reported in Tab.~\ref{tab:1sm} (adapted from \cite{breen07} and originally collected during the ``Italian Household Longitudinal Survey'') where the inter-generational social mobility has been recorded on a sample of $4,343$ Italian workers in $1997$. The data take into account the gender, and thus we have separate tables for men and women. There are $4$ categories of workers. $A$: ``High level professionals''; $B$: ``Employees and commerce''; $C$: ``Skilled working class and artisans''; $D$: ``Unskilled working class''. In \cite{breen07} these data are analyzed extensively with a thorough presentation of a lot of models to describe special patterns of mobility. Here we merely use the simplified version displayed in Tab.~\ref{tab:1sm} to show the practical applicability of the methodology introduced in Sect.~\ref{compa} also in this context.

\begin{table}[htbp]
\caption{\label{tab:1sm}Table of social mobility in Italy (1997). Columns represent the father's occupation, rows represent the son's (or daughter's) occupation. Male respondents in the left panel, female respondents in the right panel.}
\centering
\begin{tabular}{ccccccccccc}
\hline\noalign{\smallskip}
 & $A$ & $B$ & $C$ & $D$ & $\qquad$ & & $A$ & $B$ & $C$ & $D$    \\
\noalign{\smallskip}\hline\noalign{\smallskip}
$A$ & 172 & 31 & 31 & 28 & & $A$  & 137   & 52  & 29  & 15 \\
$B$ & 108 & 49 & 24 & 46 & & $B$  & 78    & 46  & 14  & 23 \\
$C$ & 174 & 84 & 301 & 272 & & $C$ & 142 & 100  & 124  & 145 \\
$D$ & 225 & 148 & 236 & 664 & & $D$ & 164 & 181 & 141  & 35 \\
\noalign{\smallskip}\hline\noalign{\smallskip}
\end{tabular}
\end{table}

We use again the quasi-symmetry model, and the Markov basis ${\mathcal M}_0$ for the model $M_0$ is formed by $2 \cdot 200$ moves. If we consider the two layers separately, we obtain exact $p$-values are equal to $0.051$ and $0.088$ respectively ($G^2= 6.703$ and $G^2=8.279$ respectively, with $3$ df). Although this fit may appear weak, nevertheless the situation dramatically changes when considering the test for nested models. Running the test for a unique quasi-symmetry model, the exact test produces a $p$-value equal to $0$ ($G^2=112.687$ with $12$ df, corresponding to an asymptotic $p$-value less than $10^{-15}$), meaning that there is a strong departure from the null hypothesis. Combining these results, one can conclude that the two genders present strong differences in terms of patterns of mobility.


\end{document}